\documentclass[11pt]{amsart}
\usepackage[left=3.8cm,right=3.8cm]{geometry}

\usepackage[english]{babel}
\usepackage[utf8]{inputenc}

\usepackage{amsmath,amsfonts,amssymb,amscd,amsthm}
\usepackage[linktocpage,unicode]{hyperref}

\usepackage{latexsym,graphics}
\usepackage[dvips]{color}
\usepackage[dvips]{graphicx}
\usepackage{pstricks,pst-tree,rotating}

\title[A generalization of Euler numbers to finite Coxeter groups]
      {A generalization of Euler numbers to finite Coxeter groups}
\date{}
\author{Matthieu Josuat-Vergès}
\thanks{Supported by ANR CARMA (ANR-12-BS01-0017).}
\address{CNRS and Institut Gaspard Monge, Université Paris-Est Marne-la-Vallée\\
5 Boulevard Descartes\\
Champs-sur-Marne\\
77454 Marne-la-Vallée cedex 2\\ France}
\email{matthieu.josuat-verges@univ-mlv.fr}

\hypersetup{
    unicode=true,          
    pdftoolbar=true,        
    pdfmenubar=true,        
    pdffitwindow=false,     
    pdfstartview={FitH},    
    pdftitle={A generalization of Euler numbers to finite Coxeter groups},
    pdfauthor={Matthieu Josuat-Vergès},
    pdfsubject={},
    pdfkeywords={},
    pdfnewwindow=true,      
    colorlinks=true,       
    linkcolor=red,          
    citecolor=blue,        
    filecolor=magenta,      
    urlcolor=cyan           
}

\newtheorem{theo}{Theorem}[section]

\newtheorem{prop}[theo]{Proposition}

\theoremstyle{definition}
\newtheorem{defi}[theo]{Definition}
\newtheorem{rema}[theo]{Remark}

\DeclareMathOperator{\stab}{Stab}
\DeclareMathOperator{\fix}{Fix}

\begin{document}

\begin{abstract}
It is known that Euler numbers, defined as the Taylor coefficients of the tangent and secant functions, 
count alternating permutations in the symmetric group. Springer defined a generalization of these numbers for 
each finite Coxeter group by considering the largest descent class, and computed the value in each case of 
the classification. 
We consider here another generalization of Euler numbers for finite Coxeter groups, building on Stanley's result 
about the number of orbits of maximal chains of set partitions. We present a method to compute these integers and 
obtain the value in each case of the classification.
\end{abstract}

\maketitle

\section{Introduction}

It is known since long ago \cite{andre} that the Euler numbers $T_n$, defined by 
\begin{equation} \label{deftn}
  \sec(z)+\tan(z) = \sum_{n\geq 0} T_n \frac{z^n}{n!}, 
\end{equation}
count alternating permutations in the symmetric group $\mathfrak{S}_n$
($\sigma$ is {\it alternating} if $\sigma(1)>\sigma(2)<\sigma(3)>\dots$).
Since then, there has been a lot of interest in these numbers and permutations, as exposed
in the recent survey of Stanley \cite{stanley}. 

It can be shown that alternating permutations form the largest descent class in the symmetric group.
Building on this, Springer \cite{springer} gave a characterization of the largest descent class of 
a finite Coxeter group, and computed its cardinality in each case of the classification.
The analog of alternating permutations for other groups were studied by Arnol'd \cite{arnold}, 
who called these objects {\it snakes}.
See also \cite[Section~3]{jnt} for an alternative proof of Springer's result. 
Another relevant reference is Saito's article \cite{saito}, where a more general problem is considered.

%
%
%

In this article, we are interested in another construction that relates the number $T_n$ with the symmetric group,
and can also be generalized to finite Coxeter groups.
Namely, there is an action of $\mathfrak{S}_n$ on the maximal chains in the lattice
of set partitions of size $n$, and Stanley \cite{stanley2} showed that the number of orbits is $T_{n-1}$.
It is now well established that set partitions can be realized as an intersection lattice generated by 
reflecting hyperplanes, so that the construction can be generalized and gives an integer $K(W)$ for each 
finite Coxeter group $W$, with $K(A_n)=T_n$. (Note that this differs from Springer's construction, where 
the the integer $T_n$ is related with the group $A_{n-1}$.)
We present a general method to compute $K(W)$ and apply it to obtain the value in each case of the classification.
There is some similarity with a problem studied by Reading \cite{reading}, which consists in the enumeration of 
maximal chains in the lattice of noncrossing partition (in both cases, there is a product formula for the 
reducible case and a recursion on maximal parabolic subgroup in the irreducible case). 

\section*{Acknowledgement}

I thank Vic Reiner for his helpful comments about this work.
I also thank the anonymous referees for helpful suggestions and corrections.

\section{\texorpdfstring{Definitions}{Definitions}}

\label{def}


Let $V$ be a Euclidean space, and $W$ 
a finite subgroup of $GL(V)$ generated by orthogonal reflections. Let $n$ be the rank of $W$, i.e. $n=\dim V$.
We call {\it reflecting hyperplane} an hyperplane $H\subset V$ which is the fixed point set of some 
reflection in $W$. The following definition is now 
well established, see for example \cite[Chapter 4]{armstrong}.

\begin{defi}
  The {\it set partition lattice} $\mathcal{P}(W)$ is the set of linear subspaces of $V$ that
  are an intersection of reflecting hyperplanes. It is ordered by reverse inclusion, i.e. $\pi \leq \rho$
  if $\rho \subset \pi$.
\end{defi}

\begin{rema}
We are mostly interested in the case where $V$ is the standard geometric representation of a finite Coxeter group $W$.
In this case, $\{0\} \in \mathcal{P}(W)$ and it is the maximal element.
But in what follows, it will also be convenient to consider some reflection subgroup $U\subset W$. The definition
is still valid and gives a subset $\mathcal{P}(U)\subset \mathcal{P}(W)$, and $\{0\} \notin \mathcal{P}(U)$ {\it a priori}.
\end{rema}

In the case $A_n$ of the classification, $W$ is the symmetric group $\mathfrak{S}_{n+1}$ acting on 
$V=\{ v\in\mathbb{R}^{n+1} \, : \, \sum v_i=0 \}$ by permuting coordinates. The reflecting 
hyperplanes are $H_{i,j} = \{ v\in V \, : \, v_i=v_j \} $ where $i<j$. We recover the traditional definition
of a set partition, for example if $n=6$, then 
\[ H_{1,7} \cap H_{2,4} \cap H_{4,5} = \{ v\in V \,:\, v_1=v_7, \; v_2=v_4=v_5 \} \in \mathcal{P}(A_{6}) 
 \]
corresponds to the set partition $17|245|3|6$.


Let $t\in W$ be a reflection, and $H = \fix(t)$  be its fixed point set.
Then $w(H) = \fix( w t w^{-1} )$ for $w\in W$.
So the natural action of $W$ on linear subspaces of $V$ gives an action of $W$ on the reflecting hyperplanes, 
and on $\mathcal{P}(W)$. Since inclusion and rank are preserved, this extends to an action on the maximal chains 
in $\mathcal{P}(W)$.

\begin{defi}
 Let $\mathcal{M}(W)$ denote the set of maximal chains in $\mathcal{P}(W)$, i.e. sequences $C=(C_0,\dots,C_n) \in \mathcal{P}(W)^{n+1}$
 where $C_0<\dots<C_n$ (this implies that $C_i$ has rank $i$).
 We define an integer $K(W)$ as the number of orbits for the $W$-action on $\mathcal{M}(W)$, i.e. $K(W) = \# ( \mathcal{M}(W) / W  ) $.
\end{defi}

An element of $\mathcal{M}(W)$ can be seen as a complete flag of $V$. Thus we can rephrase the definition:
$K(W)$ is the number of $W$-orbits of complete flags in $V$ where each element of the flag is a fixed point subspace
of some $w\in W$. 

Let us introduce further notations (see \cite{bjorner,humphreys}).
We recall that the complement in $V$ of the reflecting hyperplanes is divided into connected regions called {\it chambers},
and $W$ acts simply transitively on the chambers. Let $H_1,\dots,H_n$ be the reflecting hyperplanes that enclose 
one particular chamber $R_0$, the fundamental chamber.
Then the corresponding orthogonal reflections $s_1,\dots,s_n$ form a set $S$ 
of simple generators for $W$. According to this choice, there is a longest element $w_0$ (the unique group element
that maximizes the length function). 
For any $i$, let $W_{(i)} \subset W$ be the (standard maximal parabolic) subgroup generated by the $s_j$ with $j\neq i$.
If $s\in S$, we also denote $W_{(s)}=W_{(i)}$ if $s=s_i$.
An alternative description is that, if we define a line
\begin{equation} \label{defli}
  L_i = \bigcap\limits_{ \substack{ 1\leq j \leq n \\[1mm] j\neq i } } H_j ,
\end{equation}
then $w\in W_{(i)}$ if and only if $w(v)=v$ for any $v\in L_i$.
The lines $L_i$ are exactly those in $\mathcal{P}(W) $ that are incident to the fondamental chamber $R_0$.


For each line $L\in\mathcal{P}(W)$, we define two subgroups of $W$, 
respectively the stabilizer and the pointwise stabilizer: 
\begin{align*}
   \stab(L) &= \big\{ w\in W \, : \, w(L)=L \big\}, \\
   \stab^*(L)  &= \big\{ w\in W \, : \, \forall x \in L, \, w(x)=x  \big\}.
\end{align*}
Note that $\stab^*(L)$ is a subgroup of $\stab(L)$ with index either 1 or 2.
The group $\stab^*(L)$ is generated by the reflections it contains and is
itself a real reflection group, its reflecting hyperplanes being those 
of $W$ containing $L$. So we can identify $\mathcal{P}(\stab^*(L))$ with the interval $[V,L]\subset\mathcal{P}(W)$.

\section{\texorpdfstring{The general method}{The general method}}

\label{method}

We describe how the integer $K(W)$ can be computed inductively. To begin, in the reducible case we have:

\begin{prop} \label{product}
Let $W_1$ and $W_2$ be two Coxeter groups of respective ranks $m$ and $n$, then 
\[
  K(W_1\times W_2) = \binom {m+n}{m} K(W_1)K(W_2).
\]
\end{prop}

\begin{proof}
First, note that there is natural identification 
$\mathcal{P}(W_1)\times \mathcal{P}(W_2) = \mathcal{P}(W_1\times W_2) $.
The idea is to shuffle an element of $\mathcal{M}(W_1)$ with one of $\mathcal{M}(W_2)$ and the details are as follows.
Let $(x_0,\dots,x_m) \in \mathcal{M}(W_1)$ and $(y_0,\dots,y_n) \in \mathcal{M}(W_2)$.
By elementary properties of the product order, we can form an element $C \in \mathcal{M}(W_1\times W_2)$
by considering a sequence
\[
  C = ((x_{i_0},y_{j_0}) , \dots , (x_{i_{m+n}},y_{j_{m+n}}) )
\]
where the indices are such that $i_0=j_0=0$, $i_{m+n}=m$, $j_{m+n}=n$, and for $0\leq k <m+n$:
\begin{itemize}
 \item either $i_{k+1}=i_{k}$ and $j_{k+1}=j_k+1$, 
 \item or $i_{k+1}=i_{k}+1$ and $j_{k+1}=j_k$.
\end{itemize}
If $I$ denotes the set of possible choices for the indices $i_k$ and $j_k$, this defines a bijection 
\[
  I \times \mathcal{M}(W_1)\times \mathcal{M}(W_2) \to \mathcal{M}(W_1\times W_2).
\]
Since the bijection commutes with the action of $W_1\times W_2$ and $\# I = \binom{m+n}{m}$,
the result is proved.
\end{proof}

We suppose now that $W$ is irreducible. 
A natural approach to find $K(W)$ is to distinguish the maximal chains according to the coatom they contain 
(in terms of complete flags, we distinguish them according to the line they contain).
Doing the same thing at the level of orbits will lead to Proposition~\ref{induct} below.

Recall that we can identify $\mathcal{P}(\stab^*(L))$ with $[V,L]\subset \mathcal{P}(W)$.
There is also a natural way to see $\mathcal{M}(\stab^*(L))$ as a subset of $\mathcal{M}(W)$, namely
$(C_0,\dots,C_{n-1})\in \mathcal{M}(\stab^*(L))$ is identified with $(C_0,\dots,C_{n-1},\{0\})$.
Clearly, $[V,L]$ is stable by the action of $\stab(L)$ and this extends
to an action of $\stab(L)$ on $\mathcal{M}(\stab^*(L))$. With this at hand, we have:

\begin{prop} \label{induct}
Let $ \mathcal{L} \subset \mathcal{P}(W)$ be a set of orbit representatives for the action of $W$ on lines in $\mathcal{P}(W)$,
then:
\[
  K(W) = \sum_{ L\in\mathcal{L} }  \# \big( \mathcal{M}( \stab^*(L) ) / \stab(L) \big).
\]
\end{prop}

\begin{proof}
If $C = (C_0,\dots,C_n)\in \mathcal{M}(W)$, there is a unique $L\in\mathcal{L}$ such that the coatom $C_{n-1}$ and $L$ are in the 
same $W$-orbit. Moreover, $L$ only depends on the $W$-orbit of $C$, so this defines a map $f:\mathcal{M}(W) / W \to  \mathcal{L} $.

With the discussion above in mind, we identify $\mathcal{M}(\stab^*(L))$ with the set of chains $C=(C_0,\dots,C_n)\in\mathcal{M}(W)$ 
satisfying $C_{n-1}=L$. 
Each element of $f^{-1}(L)$ is a $W$-orbit that can be represented by an element of $\mathcal{M}(\stab^*(L))$,
and two elements of $\mathcal{M}(\stab^*(L))$ are in the same $W$-orbit if and only if they are in the same $\stab(L)$-orbit.
This permits to define a bijection between $f^{-1}(L)$ and $\mathcal{M}( \stab^*(L) ) / \stab(L)$.
Now, we can write:
\[
  K(W) = \# (\mathcal{M}(W)/W ) =  \sum_{L\in \mathcal{L} }  \# ( f^{-1}(L) ) =  \sum_{ L\in\mathcal{L} }  \# \big( \mathcal{M}( \stab^*(L) ) / \stab(L) \big),
\]
as announced.
\end{proof}

Now, let us describe how to find the set $\mathcal{L}$ of orbit representatives for the action of $W$ on lines in $\mathcal{P}(W)$. 
We can use the lines $L_i$ defined in Equation~\eqref{defli} from the previous section.

\begin{prop} \label{wli}
Each line $L\in\mathcal{P}(W)$ can be written $w(L_i)$ for some $w\in W$ and $1\leq i \leq n$.
If $w\in W$ and $i\neq j$, then $w(L_i)=L_j$ implies $w_0(L_i)=L_j$.
\end{prop}

Similar considerations appeared in the work of Armstrong, Reiner and Rhoades \cite{armstrong2},
in the context of $W$-parking functions.
Still, it is reasonable to include a short proof here.

\begin{proof}
Let us split the line $L$ in two half-lines $L^+$ and $L^-$, and let $R$ be a chamber incident to $L^+$. 
We also split $L_i$ in two half-lines $L_i^+$ and $L_i^-$, where $L_i^+$ is the one incident to $R_0$.
The group $W$ acts simply transitively on the chambers, so there is $w\in W$ such that $w(R_0)=R$.
Then $w^{-1}(L^+)$ is incident to $R_0$, so there is $i$ such that $w^{-1}(L^+)=L_i^+$, and
consequently $L^+=w(L_i^+)$ and $L=w(L_i)$.

Now, suppose we have $i\neq j$ and $w(L_i)=L_j$. 
We have either $w(L_i^+)=L_j^+$ or $w(L_i^+)=L_j^-$ (where $L_j^+$ and $L_j^-$ are defined in the same way as with $L_i$).
In the first case, $R_0$ and $w(R_0)$ are both incident to $L_j^+$. This implies $w(L_j^+)=L_j^+$
(note that $W_{(j)}$ acts simply transitively on the set of chambers incident to $L_j^+$),
but this is a contradiction with $i\neq j$ and $w(L_i)=L_j$. So we have $w(L_i^+)=L_j^-$.
Since $L_j^-$ is incident to both $-R_0$ and $w(R_0)$, there is $u\in W_{(j)}$ such that 
$uw(R_0)=-R_0$, i.e. $uw=w_0$. Then, we have $w_0(L_i^+) = uw(L_i^+) = u(L_j^-) = L_j^-$.
So $w_0(L_i)=L_j$.
\end{proof}

From the definition of $L_i$ in Equation~\eqref{defli}, $w_0(L_i)=L_j$ is equivalent to $w_0(H_i)=H_j$, 
which is also equivalent to $w_0s_iw_0=s_j$. Elementary properties of the longest element show that the 
map defined on the simple generators by $s \mapsto w_0 s w_0$ is an involutive automorphism of the Coxeter 
graph. One can also show that this automorphism is the identity if and only if the exponents of the 
group are all odd, see \cite[Exercise 4.10]{bjorner}. So the set $\mathcal{L}$ can be obtained by taking
$\{L_1,\dots,L_n\}$, quotiented by the action of $w_0$ which can be described in a precise way.


We have $\stab^*(L_i)=W_{(i)}$, the standard maximal parabolic subgroup. To identify the group $\stab(L_i)$, 
we have the following:

\begin{prop} \label{propliwi}
Either $\stab(L_i)=W_{(i)}$, or $\stab(L_i) = < W_{(i)} , w_0 >$.
\end{prop}

\begin{proof}
Suppose there is $w \in \stab(L_i)$ with $w\notin W_{(i)}$, which means that $w(L_i^+) = L_i^-$.
So $w(R_0)$ is incident to $L_i^-$. Since $W_{(i)}$ acts transitively on the chambers incident to $L_i^-$,
there is $u\in W_{(i)}$ with $uw(R_0) = -R_0$, i.e. $uw=w_0$. It follows $w_0 \in \stab(L_i)$ with $w_0 \notin W_{(i)}$. 
\end{proof}



Since $W_{(i)}$ has rank $n-1$, by induction we can assume we already know the integer $K(W_{(i)})$,
which is useful in some situations.

\begin{prop} \label{invohalf} With $W$, $w_0$, and $L_i \in \mathcal{L}$ as above, we have:
\begin{itemize}
 \item If $w_0 s_i w_0 \neq s_i$, then 
\[
  \# ( \mathcal{M}( W_{(i)}) / \stab(L_i) ) = K( W_{(i)} ).
\]
 \item If $w_0 s_i w_0 = s_i$, and there is $ u \in W_{(i)}$ such that $w_0 s_j w_0 = u s_j u $ for any $j\neq i$, then
\[
  \# ( \mathcal{M}( W_{(i)} ) / \stab(L_i) ) = K( W_{(i)} ).
\]
 \item If $w_0 s_i w_0 = s_i$, and the map $s\mapsto w_0 s w_0$ permutes nontrivially the connected components of the Coxeter graph of $W_{(i)}$, then:
\[
  \# ( \mathcal{M}(W_{(i)}) / \stab(L_i) ) = \frac 12 K( W_{(i)} ).
\]
\end{itemize}
\end{prop}

\begin{proof}

If $w_0 s_i w_0 \neq s_i$, then $w_0\notin \stab(L_i) $, hence $\stab(L_i)=W_{(i)}$ using Proposition~\ref{propliwi}.
This proves the first point.

Suppose $w_0 s_i w_0 = s_i$ and there exists $u$ as above. It means that the action of $u$ on $\mathcal{M}(W_{(i)})$
is the same as the action of $w_0$. In either of the two cases given in Proposition~\ref{propliwi}, we find that
the $\stab(L_i)$-orbits are exactly the $W_{(i)}$-orbits. This proves the second point.

As for the third point, we suppose there are only two connected components in the Coxeter graph of $W_{(i)}$, the 
general case being similar. Let us write $W_{(i)} = W_1\times W_2 $.
We have seen in the proof of Proposition~\ref{product} that the elements of $\mathcal{M}(W_{(i)})$ are obtained
by ``shuffling'' two elements of $\mathcal{M}(W_1)$ and $\mathcal{M}(W_2)$.
So if $C=(C_0,\dots,C_{n-1}) \in \mathcal{M}(W_{(i)})$, the element $C_1$ is a pair 
$(C_1', C_1'') \in \mathcal{P}(W_1)\times \mathcal{P}(W_2)$ where
the respective ranks of $C_1'$ and $C_1''$ are either 0 and 1, or 1 and 0.
These two conditions are preserved by the action of $W_{(i)}$, and are reversed by the action of $w_0$.
So the action of $w_0$ on $\mathcal{M}( W_{(i)} ) / W_{(i)} $ has no fixed point and each orbit has cardinality 2.
We can write:
\[
  \# ( \mathcal{M}(W_{(i)}) / \stab(L_i) ) = \# (( \mathcal{M}(W_{(i)}) / W_{(i)} ) / w_0 )
\]
and this proves the result.
\end{proof}

Let us summarize the situation. If $w_0$ is central in $W$, 
we can always apply the second case of Proposition~\ref{invohalf}, so that Proposition~\ref{induct} gives
\begin{equation} \label{summ1}
  K(W) = \sum_{s\in S} K(W_{(s)}),
\end{equation}
where each $W_{(s)}$ is a standard maximal parabolic subgroup of $W$. Furthermore, some of the terms are simplified 
using the product formula in Proposition~\ref{product}. In particular, this equation can be directly obtained from the Coxeter graph.

When $w_0$ is not central, the map $s\mapsto w_0 s w_0 $ is an involution on the set $S$ of simple generators and we need to 
distinguish the two-element orbits and the fixed points. Indeed, we have:
\begin{equation} \label{summ2}
  K(W) = \sum_{\substack{ \{s_i,s_j\} \subset S, \; s_i\neq s_j  \\[1mm]  w_0 s_i w_0 = s_j  } } K(W_{(i)}) + 
         \sum_{\substack{ s_i\in S  \\[1mm]  w_0 s_i w_0 = s_i } } \#( \mathcal{M}( W_{(i)} ) / \stab(L_i) ).
\end{equation}
Some terms in the first sum (respectively, the second sum) can be further simplified using Proposition~\ref{product}
(respectively, Proposition~\ref{invohalf}). 

Note that Proposition~\ref{invohalf} does not exhaust all the possibilities, so we do not have a general solution to find 
all the terms $\#( \mathcal{M}( W_{(i)} ) / \stab(L_i) )$ in the second sum of Equation~\eqref{summ2}. As we will
see in the next section, the only case that cannot be treated directly will appear when $W=D_n$ with $n$ odd.

\section{The case by case resolution}

\label{results}

We follow the traditional notation for the classification of finite irreducible Coxeter groups, see \cite{bjorner}.
We will denote $a_n=K(A_n)$, $b_n=K(B_n)$, $d_n=K(D_n)$. 
It will be convenient to take the conventions that $A_0=B_0=D_0$
(the trivial group with rank 0), $A_1=B_1$, $D_2=A_1\times A_1$ and $D_3=A_3$.

\begin{prop}[See \cite{bjorner}, Exercise 4.10]  \label{involcase}
In the groups $I_2(m)$ for $m$ even, $B_n$, $D_{n}$ for $n$ even, $G_2$, $H_3$, $H_4$, $E_7$, and $E_8$, the longest element is central.
In the other groups, i.e. $I_2(m)$ for $m$ odd, $A_n$, $D_n$ for $n$ odd, and $E_6$, the map $s\mapsto w_0 s w_0$ is the unique 
nontrivial automorphism of the Coxeter graph.
\end{prop}

\subsection{\texorpdfstring{Case of $A_n$}{Case of A n}}

We already know that $a_n=T_n$, but let us check how to prove it with our method.
Here, $w_0$ is not central and $s\mapsto w_0sw_0$ reverses the $n$ vertices of the Coxeter graph.
There is a fixed point only if $n$ is odd, and it can be treated using the third case of Proposition~\ref{invohalf}.
So Equation~\eqref{summ2} gives, when $n\geq 2$:
\[
  a_n = \sum_{i=0}^{\lfloor n/2 \rfloor - 1 } \binom{n-1}{i} a_i a_{n-1-i} 
        + [n\text{ mod }2] \times \frac 12\binom{n-1}{(n-1)/2} a_{(n-1)/2}^2.
\]
(Here, $[n\text{ mod }2]$ is considered as the natural number 0 or 1.)
This can be rewritten as:
\begin{equation} \label{recan}
  a_n = \frac12 \sum_{i=0}^{n-1} \binom{n-1}{i} a_i a_{n-1-i}.
\end{equation}
Let us define
\[
  A(z) = \sum_{n\geq 0} a_n \frac{z^n}{n!}.
\]
Multiplying Equation~\eqref{recan} by $\frac{z^{n-1}}{(n-1)!}$ and summing over $n\geq2$ gives 
\[ 
  A'(z)-1 = \frac12( A(z)^2-1 ).  
\]
So $A(z)$ is the solution of the differential equation $A'(z) = \frac 12 (A(z)^2+1) $ with the initial value $A(0)=1$.
It can be checked that $A(z)=\tan(z)+\sec(z)$ is the solution, so that $a_n=T_n$.

\subsection{\texorpdfstring{Case of $B_n$}{Case of B n}}

In this group, the longest element is central. Equation~\eqref{summ1} together with the product formula gives:
\begin{equation} \label{recbn}
  b_n = \sum_{i=0}^{n-1}  \binom{n-1}i b_i a_{n-i-1}.
\end{equation}
Now, let 
\[
  B(z) = \sum_{n\geq 0} b_n \frac{z^n}{n!}. 
\]
Multiplying Equation~\eqref{recbn} by $\frac{z^{n-1}}{(n-1)!}$ and summing over $n\geq 1$ gives 
\[
  B'(z) = B(z) A(z).
\]
So $B(z)$ is the solution of the differential equation $B'(z)=B(z)A(z)$ with initial value $B(0)=1$. We can check that 
\[
  B(z) = \frac{1}{1-\sin(z)}
\]
is a solution. This function also satisfies $B(z)=A'(z)$, so that 
\[
  b_n = T_{n+1}. 
\]
A bijective proof of this will be given in \cite{josuat}.

\subsection{\texorpdfstring{Case of $D_n$}{Case of D n}}


When $n$ is even, the longest element of $D_n$ is central and Equation~\eqref{summ1} gives:
\begin{equation} \label{recdneven}
  d_n = 2 a_{n-1} + \sum_{ 2\leq i \leq n-1 } \binom{n-1}{i}  d_i a_{n-1-i}.
\end{equation}

In the case when n is odd, one cannot quite write the equation as immediately.
The map $s\mapsto w_0 s w_0$ exchanges two vertices of the Coxeter graph, and
this gives one term $a_{n-1}$ coming from the first sum in Equation~\eqref{summ2}.
As for the second sum, we are in the case where $s_i = w_0 s_i w_0$, 
and $W_{(i)}=D_i\times A_{n-1-i}$.
If $i$ is odd, we can apply the second case of Proposition~\ref{invohalf} where $u$ is 
chosen to be the longest element of the factor $D_i$.
More care is needed when $i$ is even, i.e. when we cannot directly apply Proposition~\ref{invohalf}.
So we consider the set $\mathcal{M}(D_i \times A_{n-1-i})$, quotiented by $D_i \times A_{n-1-i}$,
and further quotiented by the graph automorphism of the factor $D_i$ (the graph automorphism induces an action
on $\mathcal{P}(D_i)$).
An argument similar to
the one in Proposition~\ref{product} shows that the number of orbits can be factorized. Eventually, we obtain:
\begin{equation} \label{recdnodd}
  d_n = a_{n-1} + \sum_{ \substack{ 2\leq i \leq n-1 \\[1mm] i \text{ odd } }} \binom{n-1}{i}  d_i a_{n-1-i} 
                 + \sum_{ \substack{ 2\leq i \leq n-1 \\[1mm] i \text{ even } }} \binom{n-1}{i}  \bar d_i a_{n-1-i},
\end{equation}
where $\bar d_i$ is defined as follows: it is the number of orbits for the action on $\mathcal{M}(D_i)$ generated by 
$D_i$ together with the graph automorphism 
(except that if $i=4$, the graph automorphism is not unique but we only consider the one that exchanges two vertices).
Note that for odd $i$, we can define $\bar d_i$ similarly but it is clear that $\bar d_i = d_i$.
We need to compute $\bar d_n$ before solving the recursion for $d_n$. 

\begin{prop}
We have $\bar d_0=1$ and for any $n\geq 1$,
\begin{equation} \label{recbardn}
  \bar d_n = a_{n-1} + \sum_{i=2}^{n-1} \binom{n-1}{i} \bar d_i a_{n-1-i}.
\end{equation}
\end{prop}

\begin{proof}
Although we cannot directly apply Proposition~\ref{induct} and Proposition~\ref{product}, the argument is completely similar,
so we omit details. Let $\Gamma$ denote the graph automorphism of $D_n$.

Suppose $L_1$ and $L_2$ are the two coatoms that are exchanged by $\Gamma$. Counting orbits of maximal 
chains having $L_1$ or $L_2$ as coatom, we obtain the first term $a_{n-1}$.

If $i\neq 1,2$, the number of orbits of maximal chains having $L_i$ as coatom is the number of orbits in 
$\mathcal{M}(W_{(i)}) / <  \stab(L_i) , \Gamma > $. This is also the number of orbits in 
$\mathcal{M}(W_{(i)}) / <  W_{(i)} , \Gamma > $, since either $\stab(L_i)=W_{(i)}$ or 
$\stab(L_i)=<W_{(i)},w_0>$ where $w_0$ has the same action as $\Gamma$.
We have a decomposition $W_{(i)} = D_i \times A_{n-1-i}$ and the graph automorphism only acts on the factor 
$D_i$. So the argument of Proposition~\ref{product} shows that this number is $\bar d_i a_{n-1-i}$.
\end{proof}

\begin{prop}
If $n\geq2$, we have $\bar d _n = 2a_{n+1} - (n+1)a_n $.
\end{prop}

\begin{proof}
The recursion in the previous proposition shows that the generating function 
$\bar D(z) = \sum_{n\geq 0} \bar d_n \frac{z^n}{n!}$ satisfies the differential equation
\[
  \bar D'(z) = (\bar D(z)-z)A(z),
\]
with the initial condition $\bar D(0)=1$. This is solved by
\begin{equation} \label{bardz}
  \bar D(z) =  \frac{2-\cos(z)-z\sin(z)}{1-\sin(z)}.
\end{equation}
From this expression, we can get $\bar D(z) = (2-z)A'(z) + z - A(z)$, 
and it follows that $\bar d _n = 2a_{n+1} - (n+1)a_n $ if $n \geq 2$. 
\end{proof}

\begin{prop} 
Let $n\geq 2$. We have $d_n - \bar d_n = a_n$ if $n$ is even, and $d_n = \bar d_n$ otherwise.
\end{prop}

\begin{proof}
If $n\geq 2$, from \eqref{recdneven}, \eqref{recdnodd}, and \eqref{recbardn}, we have:
\[
  d_n - \bar d_n = \chi[ n \text{ even}] \times a_{n-1} + \sum_{\substack{  2\leq i \leq n-1 \\[1mm] n-i \text{ even }  }} \binom{n-1}{i} (d_i-\bar d_i) a_{n-1-i}.
\]
Here and in the sequel, $\chi$ means 1 or 0 depending on whether the condition within brackets is true or false.
So the generating function
\[
  U(z) = 1 + \sum_{ n\geq 2 }  ( d_n - \bar d_n ) \frac{z^n}{n!}
\]
satisfies $U'(z) = U(z)\tan(z) $ and $U(0)=1$. This is solved by $U(z)=\sec(z)$
and the result follows.
\end{proof}

From the previous two propositions, we get that for $n\geq 2$,
\[
  d_n = \begin{cases}
           2T_{n+1} - n T_n & \text{ if } n \text{ is even,} \\
           2T_{n+1} - (n+1) T_n & \text{ if } n \text{ is odd.}
        \end{cases}
\]
From \eqref{bardz}, we can separate the odd and even parts of $\bar D(z)$ (multiply the numerator and denominator by $1+\sin(z)$
and separate terms in the numerator). After some calculation, this leads to:
\begin{equation*}
 \sum_{n\geq 1}  d_{2n} \frac{z^{2n}}{(2n)!} = \frac{ \sin(z) (2\sin(z) - z) }{  \cos(z)^2  },
\end{equation*}
and
\begin{equation*}
 \sum_{n\geq 1}  d_{2n+1} \frac{z^{2n+1}}{(2n+1)!} = \frac{ \sin(z)(2-\cos(z)) -z }{ \cos(z)^2 }.
\end{equation*}
We can take the sum of these two equations to obtain $\sum_{n\geq 2} d_n \frac{z^n}{n!}$,
but there seems to be no particular simplification.
The first values of $\bar d_n$ for $n\geq 2$ are as follows:
\[
  1,2,7,26,117,594,3407,21682,151853,1160026,9600567...
\]
And the first values of $d_n$ for $n\geq 2$ are:
\[
  2,2,12,26,178,594,4792,21682,202374,1160026,12303332,...
\]

\subsection{Remaining cases}

For the dihedral group, we have:
\[
  K( I_2(m) ) = \begin{cases}
                    1 \text{ if } m \text{ is odd,} \\
                    2 \text{ if } m \text{ is even.}
                \end{cases}
\]

Among the exceptional groups, $E_6$ is the only one where the longest element is not central.
We apply Equation~\eqref{summ2} and the calculation is the following:
\begin{align*}
  K(E_6) &=  K(D_5) + K(A_4\times A_1) + \frac 12 K(A_2\times A_1 \times A_2 ) + K(A_5) \\
         &= 26+25+15+16 = 82.
\end{align*}
The first two terms correspond to the terms where $s_i\neq s_j$ and $w_0s_iw_0=s_j$.
The third term corresponds to a fixed point of the graph automorphism, the vertex of degree 3.
It is treated using the second part of Proposition~\ref{invohalf}.
The fourth term corresponds to the other fixed point of the graph automorphism,
it is treated using the first part of Proposition~\ref{invohalf}.

For all the remaining groups, the longest element is central and we can apply Equation~\eqref{summ1}. This gives:
\begin{align*}
  K(H_3) &= K( I_2(5) ) +  K( A_1\times A_1 ) + K( A_2 ) = 4,  \\
  K(H_4) &=  K( H_3 ) + K( I_2(5) \times A_1 ) + K( A_2\times A_1 ) + K( A_3) = 12, \\
  K(F_4) &= K(B_3) + K(A_2\times A_1) + K(A_1\times A_2) + K(B_3) =  16. 
\end{align*}
Eventually, we have:
\begin{align*}
  K(E_7) &= K(E_6) + K(D_5\times A_1) + K( A_4\times A_2 ) +  \\
           & \qquad K(A_3\times A_1\times A_2) +  K(A_1\times A_5) + K(D_6) + K(A_6) \\
         &= 82 +  156 + 75 + 120 + 96  + 178 + 61 = 768,
\end{align*}
and
\begin{align*}
  K(E_8) &=  K(E_7) + K(E_6\times A_1) + K(D_5\times A_2) + K(A_4\times A_3) + \\
           & \qquad K(A_2\times A_1 \times A_4) + K(A_6\times A_1) + K(D_7) + K(A_7) \\
         &= 768 + 574 + 546 + 350 + 525 + 427 + 594 + 272 = 4056.
\end{align*}

%

\section{Final remarks }

\label{gener}

Let us briefly mention some related results that will appear in \cite{josuat}.
Let $c$ be a Coxeter element for $W$, and consider the set of {\it noncrossing partitions}
$\mathcal{P}^{NC}(W,c)$ (see \cite{armstrong} for background on noncrossing partitions).
This sets naturally embeds in $\mathcal{P}(W)$. It is not stable under the action of $W$, but
we can consider how it is divided in equivalence classes (each class is the intersection 
of $\mathcal{P}^{NC}(W,c)$ with a $W$-orbit). There are $K(W)$ equivalence classes.
We will show in \cite{josuat} that 
we can in some sense compute the cardinality of each class, and that this leads to 
hook length formulas in type A and B.

Let us end this article which a more general question, which is not very precisely stated.
Let $G$ be a subgroup of $GL(\mathbb{R}^n)$ (there are probably some restrictions to consider,
see below). We can define a set
\begin{align*}
  \mathcal{P}(G) &= \big\{ \pi \subset \mathbb{R}^n \, : \,  \exists g\in G, \, \pi=\fix(g)     \big\},
\end{align*}
and let $\mathcal{M}(G)$ denote the set of maximal chains in $\mathcal{P}(G)$ with respect to inclusion.
The group $G$ acts on $\mathcal{P}(G)$ and $\mathcal{M}(G)$, and we can define
let $K(G) = \# (\mathcal{M}(G)/G) $. We have examined the case where $G$ is a finite reflection group 
but we see that the definition is valid in a more general context.
Note that a natural restriction on the group $G$ is the requirement that $\mathcal{M}(G)$ is a set
of complete flags.
Suppose for example that $G$ is the set of invertible upper-triangular matrices. Then $\mathcal{P}(G)$ is the set of all 
linear subspaces of $\mathbb{R}^n$, as can be seen using the $LU$ decomposition. So $\mathcal{M}(G)$ is 
the complete flag variety $GL(\mathbb{R}^n) / G $. Using the Bruhat decomposition, we see that $K(G)=n!$.
It might be of interest to examine the case of other groups.

\setlength{\parindent}{0mm}


\begin{thebibliography}{xxx}

\bibitem{andre}
 D. André:
 Développement de sec $x$ and tg $x$, 
 C. R. Math. Acad. Sci. Paris 88 (1879), 965--979.


\bibitem{arnold}
 V.I. Arnol'd:
 The calculus of snakes and the combinatorics of Bernoulli, Euler, and Springer numbers for Coxeter groups. 
 Russian Math. Surveys 47 (1992), 1--51.


\bibitem{armstrong}
 D. Armstrong:
 Generalized Noncrossing Partitions and Combinatorics of Coxeter Groups.
 Mem. Amer. Math. Soc. 202, 2009.

\bibitem{armstrong2}
 D. Armstrong, V. Reiner and B. Rhoades:
 Parking spaces. Preprint (2012), available at \href{http://arxiv.org/abs/1204.1760}{arXiv:1204.1760}.

\bibitem{bjorner}
 A. Björner and F. Brenti:
 Combinatorics of Coxeter groups.
 Graduate Texts in Math., Vol. 231, Springer, New York, 2005.

%
%
%
%
%
%


\bibitem{humphreys}
 J.E. Humphreys: Reflection Groups and Coxeter Groups. 
 Cambridge University Press, Cambridge, 1990.


\bibitem{josuat}
 M. Josuat-Vergès,
 Refined enumeration of noncrossing chains and hook formulas,
 in preparation.

\bibitem{jnt}
 M. Josuat-Vergès, J.-C. Novelli, and J.-Y. Thibon:
 Algebraic combinatorics on snakes.
 J. Combin. Theory Ser. A. 119(8) (2012), 1613--1638.




\bibitem{reading}
 N. Reading:
 Chains in the noncrossing partition lattice. 
 SIAM J. Discrete Math. 22(3) (2008), 875--886.

\bibitem{saito}
 K. Saito: Principal $\Gamma$-cone for a tree, 
 Adv. Math. 212(2) (2007), 645--668.

\bibitem{springer}
 T.A. Springer:
 Remarks on a combinatorial problem.
 Nieuw. Arch. Wisk. 19(3) (1971), 30--36. 

\bibitem{stanley}
 R.P. Stanley: 
 Some aspects of groups acting on finite posets. 
 J. Combin. Theory Ser. A 32 (1982), 132--161.

\bibitem{stanley2}
 R.P. Stanley:
 A survey of alternating permutations.
 Contemp. Math. 531 (2010), 165--196. 


\end{thebibliography}
\end{document}